\documentclass[12pt,a4paper]{article}
\usepackage[utf8]{inputenc}
\usepackage[T1]{fontenc}

\usepackage{amsmath,amsthm} 
\usepackage{latexsym,amssymb}
\usepackage{graphicx, epsfig}
\usepackage{mathtools}
\usepackage{array,booktabs}
\usepackage{nomencl}

\usepackage{tikz}
\usetikzlibrary{graphs,graphs.standard}
\usepackage{tikz-cd}

\newcommand{\graa}{\tikz\draw[fill=black!100,thick] (0,0)circle(0.025)--(0,-0.3)circle(0.05);} 
\newcommand{\graab}{\tikz\draw[fill=black!100,thick] (0,0)circle(0.025)--(0,-0.3)circle(0.025);} 
\newcommand{\grbb}{\tikz\draw[fill=black!100,thick](0,-0.3)circle(0.05)--(0,0)circle(0.025)--(0.3,0)circle(0.025);} 
\newcommand{\grcc}{\tikz\draw[fill=black!100,thick](0,-0.3)circle(0.025)--(0,0)circle(0.05)--(0.3,0)circle(0.025);} 
\newcommand{\gree}{\tikz\draw[fill=black!100,thick](0,-0.3)circle(0.025)--(0,0)circle(0.05)--(0.3,0)circle(0.025)--(0,-0.3)circle(0.025);} 
\newcommand{\grff}{\tikz\draw[fill=black!100,thick](0,-0.3)circle(0.025)--(0,0)circle(0.05)--(0.3,0)circle(0.025)--(0.3,-0.3)circle(0.025)--(0,-0.3)circle(0.025);} 
\newcommand{\grgg}{\tikz\draw[fill=black!100,thick](0,-0.3)circle(0.025)--(0,0)circle(0.05)--(0.3,0)circle(0.025)--(0,-0.3)circle(0.025)--(0.3,0)circle(0.025)--(0.3,-0.3)circle(0.025)--(0,-0.3)circle(0.025);} 
\newcommand{\grhh}{\tikz\draw[fill=black!100,thick](0,-0.3)circle(0.05)--(0,0)circle(0.025)--(0.3,0)circle(0.025)--(0,-0.3)circle(0.025)--(0.3,0)circle(0.025)--(0.3,-0.3)circle(0.025)--(0,-0.3)circle(0.025);} 
\newcommand{\grjj}{\tikz\draw[fill=black!100,thick](0,-0.3)circle(0.025)--(0,0)circle(0.025)--(0.3,0)circle(0.025)--(0,-0.3)circle(0.025)--(0.3,0)circle(0.025)--(0.3,-0.3)circle(0.025);} 
\newcommand{\grkk}{\tikz\draw[fill=black!100,thick](0,-0.3)circle(0.05)--(0,0)circle(0.025)--(0.3,0)circle(0.025)--(0,-0.3)circle(0.025)--(0.3,0)circle(0.025)--(0.3,-0.3)circle(0.025);} 
\newcommand{\grll}{\tikz\draw[fill=black!100,thick](0,-0.3)circle(0.025)--(0,0)circle(0.025)--(0.3,0)circle(0.05)--(0,-0.3)circle(0.025)--(0.3,0)circle(0.025)--(0.3,-0.3)circle(0.025);} 
\newcommand{\grmm}{\tikz\draw[fill=black!100,thick](0,-0.3)circle(0.025)--(0,0)circle(0.025)--(0.3,0)circle(0.025)--(0,-0.3)circle(0.025)--(0.3,0)circle(0.025)--(0.3,-0.3)circle(0.05);} 
\newcommand{\groo}{\tikz\draw[fill=black!100,thick](0,-0.3)circle(0.025)--(0,0)circle(0.025)--(0.3,0)circle(0.025)--(0.3,-0.3)circle(0.05);} 
\newcommand{\grpp}{\tikz\draw[fill=black!100,thick](0,-0.3)circle(0.025)--(0,0)circle(0.025)--(0.3,0)circle(0.05)--(0.3,-0.3)circle(0.025);} 
\newcommand{\grqq}{\tikz\draw[fill=black!100,thick](0,-0.3)circle(0.025)--(0,0)circle(0.025)--(0.3,0)circle(0.05)--(0.3,-0.3)circle(0.025)--(0,0)circle(0.025)--(0.3,0)circle(0.025)--(0,-0.3)circle(0.025)--(0.3,-0.3)circle(0.025);} 
\newcommand{\grtt}{\tikz\draw[fill=black!100,thick](0,-0.3)circle(0.025)--(0,0)circle(0.05)--(0.3,0)circle(0.025)--(0,0)circle(0.025)--(0.3,-0.3)circle(0.025);} 
\newcommand{\gruu}{\tikz\draw[fill=black!100,thick](0,-0.3)circle(0.05)--(0,0)circle(0.025)--(0.3,0)circle(0.025)--(0,0)circle(0.025)--(0.3,-0.3)circle(0.025);} 

\newcommand{\griin}{
\begin{tikzpicture}
\draw[fill=black] (-3,-3) circle (1.5pt);
\draw[fill=black] (-2,-2) circle (1.5pt);
\draw[fill=black] (-1,-1) circle (1.5pt);
\draw[fill=black] (0,0) circle (1.5pt);
\draw[fill=black] (1,-1) circle (1.5pt);
\draw[fill=black] (2,-2) circle (1.5pt);
\draw[fill=black] (3,-3) circle (1.5pt);
\draw[fill=black] (-1,-3) circle (1.5pt);
\node at (-2.5,-2) {u};
\node at (1.5,-1) {v};
\draw[thick] (-3,-3) --(-2,-2)-- (-1,-3) --(-3,-3)--(-2,-2)-- (-1,-1)-- (0,0)--(1,-1)-- (-1,-1)--(1,-1)--(2,-2)--(3,-3);
\end{tikzpicture}}
\newcommand{\griio}{
\begin{tikzpicture}
\draw[fill=black] (0,0) circle (1.5pt);
\draw[fill=black] (-1,-1) circle (1.5pt);
\draw[fill=black] (-2,-2) circle (1.5pt);
\draw[fill=black] (0,-2) circle (1.5pt);
\draw[fill=black] (2,-2) circle (1.5pt);
\draw[fill=black] (1,-1) circle (1.5pt);
\draw[fill=black] (-3,-1) circle (1.5pt);
\draw[fill=black] (0,-3.4) circle (1.5pt);
\draw[fill=black] (2,0) circle (1.5pt);
\node at (1,0.3) {b};
\node at (-2.75,-1.5) {a};
\draw[thick] (-2,-2)--(0,-2)--(0,-3.4)--(0,-2)--(2,-2)--(1,-1)--(2,0)--(0,0)--(1,-1) --(0,0)--(-1,-1)--(-3,-1)--(-2,-2)--(-1,-1);
\end{tikzpicture}}

\newcommand{\griinb}{
\begin{tikzpicture}[scale=0.3]
\draw[fill=black] (0,0) circle (3pt);
\draw[fill=black] (0.577,1) circle (3pt);
\draw[fill=black] (1.155,2) circle (6pt);
\draw[fill=black] (2,3.464) circle (3pt);
\draw[fill=black] (2.845,2) circle (3pt);
\draw[fill=black] (3.423,1) circle (3pt);
\draw[fill=black] (1,0) circle (3pt);
\draw[thick] (0,0) --(1,0)-- (0.577,1)--(0,0) --(1.155,2)-- (2.845,2)-- (1.155,2)--(2,3.464)-- (2.845,2)--(3.423,1);
\end{tikzpicture}}
\newcommand{\griinc}{
\begin{tikzpicture}[scale=0.3]
\draw[fill=black] (0,0) circle (3pt);
\draw[fill=black] (0.577,1) circle (3pt);
\draw[fill=black] (1.155,2) circle (6pt);
\draw[fill=black] (2,3.464) circle (3pt);
\draw[fill=black] (2.845,2) circle (3pt);
\draw[fill=black] (1,0) circle (3pt);
\draw[thick] (0,0) --(1,0)-- (0.577,1)--(0,0) --(1.155,2)-- (2.845,2)-- (1.155,2)--(2,3.464)-- (2.845,2);
\end{tikzpicture}}
\newcommand{\griind}{
\begin{tikzpicture}[scale=0.3]
\draw[fill=black] (0,0) circle (3pt);
\draw[fill=black] (0.577,1) circle (3pt);
\draw[fill=black] (1.155,2) circle (6pt);
\draw[fill=black] (2,3.464) circle (3pt);
\draw[fill=black] (2.845,2) circle (3pt);
\draw[fill=black] (1,0) circle (3pt);
\draw[thick] (0,0) --(1,0)-- (0.577,1)--(0,0) -- (1.155,2)--(2,3.464)-- (2.845,2);
\end{tikzpicture}}
\newcommand{\griine}{
\begin{tikzpicture}[scale=0.3]
\draw[fill=black] (0,0) circle (3pt);
\draw[fill=black] (0.577,1) circle (3pt);
\draw[fill=black] (1.155,2) circle (6pt);
\draw[fill=black] (2,3.464) circle (3pt);
\draw[fill=black] (2.845,2) circle (3pt);
\draw[fill=black] (3.423,1) circle (3pt);

\draw[thick] (0,0) --(0.577,1) --(1.155,2)-- (2.845,2)-- (1.155,2)--(2,3.464)-- (2.845,2)--(3.423,1);
\end{tikzpicture}}
\newcommand{\griinf}{
\begin{tikzpicture}[scale=0.3]
\draw[fill=black] (0,0) circle (3pt);
\draw[fill=black] (0.577,1) circle (3pt);
\draw[fill=black] (1.155,2) circle (6pt);
\draw[fill=black] (2,3.464) circle (3pt);
\draw[fill=black] (2.845,2) circle (3pt);
\draw[thick] (0,0) --(0.577,1) --(1.155,2)-- (2.845,2)-- (1.155,2)--(2,3.464)-- (2.845,2);
\end{tikzpicture}}
\newcommand{\griing}{
\begin{tikzpicture}[scale=0.3]
\draw[fill=black] (0,0) circle (3pt);
\draw[fill=black] (0.577,1) circle (6pt);
\draw[fill=black] (1.155,2) circle (3pt);
\draw[fill=black] (2,3.464) circle (3pt);
\draw[fill=black] (2.845,2) circle (3pt);
\draw[thick] (0,0) --(0.577,1) --(1.155,2)-- (2.845,2)-- (1.155,2)--(2,3.464)-- (2.845,2);
\end{tikzpicture}}
\newcommand{\griinh}{
\begin{tikzpicture}[scale=0.3]
\
\draw[fill=black] (0.577,1) circle (3pt);
\draw[fill=black] (1.155,2) circle (3pt);
\draw[fill=black] (2,3.464) circle (6pt);
\draw[fill=black] (2.845,2) circle (3pt);
\draw[fill=black] (3.423,1) circle (3pt);

\draw[thick] (0.577,1) --(1.155,2)-- (2.845,2)-- (1.155,2)--(2,3.464)-- (2.845,2)--(3.423,1);
\end{tikzpicture}}
\newcommand{\griini}{
\begin{tikzpicture}[scale=0.3]
\
\draw[fill=black] (0.577,1) circle (3pt);
\draw[fill=black] (1.155,2) circle (3pt);
\draw[fill=black] (2,3.464) circle (3pt);
\draw[fill=black] (2.845,2) circle (6pt);

\draw[thick] (0.577,1) --(1.155,2)-- (2.845,2)-- (1.155,2)--(2,3.464)-- (2.845,2);
\end{tikzpicture}}
\newcommand{\griinj}{
\begin{tikzpicture}[scale=0.3]
\
\draw[fill=black] (0.577,1) circle (6pt);
\draw[fill=black] (1.155,2) circle (3pt);
\draw[fill=black] (2,3.464) circle (3pt);
\draw[fill=black] (2.845,2) circle (3pt);

\draw[thick] (0.577,1) --(1.155,2)-- (2.845,2)-- (1.155,2)--(2,3.464)-- (2.845,2);
\end{tikzpicture}}
\newcommand{\griink}{
\begin{tikzpicture}[scale=0.3]
\
\draw[fill=black] (0.577,1) circle (3pt);
\draw[fill=black] (1.155,2) circle (3pt);
\draw[fill=black] (2,3.464) circle (6pt);
\draw[fill=black] (2.845,2) circle (3pt);
\draw[fill=black] (3.423,1) circle (3pt);

\draw[thick] (0.577,1) --(1.155,2)--(2,3.464)-- (2.845,2)--(3.423,1);
\end{tikzpicture}}
\newcommand{\griinl}{
\begin{tikzpicture}[scale=0.3]
\
\draw[fill=black] (0.577,1) circle (3pt);
\draw[fill=black] (1.155,2) circle (3pt);
\draw[fill=black] (2,3.464) circle (6pt);
\draw[fill=black] (2.845,2) circle (3pt);

\draw[thick] (0.577,1) --(1.155,2)--(2,3.464)-- (2.845,2);
\end{tikzpicture}}
\newcommand{\griinm}{
\begin{tikzpicture}[scale=0.3]
\
\draw[fill=black] (0.577,1) circle (3pt);
\draw[fill=black] (1.155,2) circle (3pt);
\draw[fill=black] (2,3.464) circle (6pt);

\draw[thick] (0.577,1) --(1.155,2)--(2,3.464);
\end{tikzpicture}}
\newcommand{\griinn}{
\begin{tikzpicture}[scale=0.3]
\
\draw[fill=black] (0.577,1) circle (3pt);
\draw[fill=black] (1.155,2) circle (6pt);

\draw[thick] (0.577,1) --(1.155,2);
\end{tikzpicture}}
\newcommand{\griino}{
\begin{tikzpicture}[scale=0.3]

\draw[fill=black] (1.155,2) circle (6pt);
\draw[fill=black] (2,3.464) circle (3pt);
\draw[fill=black] (2.845,2) circle (3pt);

\draw[thick] (1.155,2)-- (2.845,2)-- (1.155,2)--(2,3.464)-- (2.845,2);
\end{tikzpicture}}

\DeclarePairedDelimiter{\card}{\lvert}{\rvert} 
\DeclareMathOperator{\Aut}{Aut} 
\DeclareMathOperator{\Orbit}{Orbit} 

\theoremstyle{plain} 
\newtheorem{thm}{Theorem}[section]

\newtheorem{lem}[thm]{Lemma}
\newtheorem{prop}[thm]{Proposition}

\theoremstyle{definition}

\newtheorem{ex}[thm]{Example}

\theoremstyle{remark}
\newtheorem{rem}[thm]{Remark}

\usepackage{authblk}

\usepackage{fullpage}

\title{A refinement of Kelly's lemma for graph reconstruction for counting
  rooted subgraphs} %
\author{Deisiane Lopes Gonçalves\thanks{Supported by CAPES, Brasil. E-mail:
    deisy.lopes28@gmail.com}\, and\, Bhalchandra D. Thatte\thanks{Supported by
    FAPEMIG, MG, Brasil, 2023-2025, Process no.  APQ-02018-22. E-mail:
    thatte@ufmg.br} \\
  Departamento de Matemática,\\
  Universidade Federal de Minas Gerais, \\
  Belo Horizonte, Brasil }%
\date{\today}
\begin{document}

\maketitle

\begin{abstract}
  Kelly's lemma is a basic result on graph reconstruction. It states that given
  the deck of a graph $G$ on $n$ vertices, and a graph $F$ on fewer than $n$
  vertices, we can count the number of subgraphs of $G$ that are isomorphic to
  $F$. Moreover, for a given card $G-v$ in the deck, we can count the number of
  subgraphs of $G$ that are isomorphic to $F$ and that contain $v$. We consider
  the problem of refining the lemma to count rooted subgraphs such that the root
  vertex coincides the deleted vertex. We show that such counting is not
  possible in general, but a multiset of rooted subgraphs of a fixed height $k$
  can be counted if $G$ has radius more than $k$. We also prove a similar result
  for the edge reconstruction problem.
\end{abstract}

\section{Introduction} The graphs in this paper are simple and finite. We denote
the vertex set of a graph $G$ by $V(G)$, the edge set by $E(G)$, and write
$v(G) = \card{V(G)}$ and $e(G) = \card{E(G)}$. The deck of a graph $G$, denoted
by $D(G)$ is the multiset of unlabelled vertex-deleted subgraphs of
$G$. Similarly, the edge deck of $G$, denoted by $ED(G)$, is the multiset of
unlabelled edge-deleted subgraphs of $G$. The well known reconstruction
conjecture of Kelly (1942) and Ulam (1960) (see
\cite{kelly.1942,kelly.1957,ulam.1960}) states that a graph on at least 3
vertices is determined up to isomorphism by its vertex-deck. The analogous edge
reconstruction conjecture of Harary \cite{harary.1964} states that a graph with
at least 4 edges is determined up to isomorphism by its edge-deck. Both these
conjectures are open till today. We refer to a survey by Bondy \cite{bondy.1991}
for more details and for the notions not defined here.

For a graph $F$, we denote by $s(F,G)$ the number of subgraphs of $G$ that are
isomorphic to $F$, and by $i(F,G)$, we denote the number of induced subgraphs of
$G$ that are isomorphic to $F$. Let $v \in V(G)$. Let $s(F,G^v)$ ($i(F,G^v)$)
denote the number of subgraphs (respectively, induced subgraphs) of $G$ that are
isomorphic to $F$ and that contain the vertex $v$. Clearly, we have
$s(F,G^v) = s(F,G) - s(F,G-v)$ and $i(F,G^v) = i(F,G) - i(F,G-v)$. A graph $G$
(or a property of $G$) is said to be reconstructible (or edge reconstructible)
if it is determined up to isomorphism (or the property is determined uniquely)
by its deck (respectively, edge deck). We have the following basic lemma for the
vertex reconstruction problem.
 
\begin{lem} \label{lem-kelly} (Kelly's Lemma) If a graph $F$ is such that
  $v(F)< v(G)$, then
  \begin{enumerate}
  \item $s(F,G)$ and $i(F,G)$  are reconstructible;
  \item $s(F,G^v)$ and $i(F,G^v)$ are reconstructible.
  \end{enumerate}
\end{lem}

In this paper, we are interested in refining the second part of the above lemma
in the following sense. If a subgraph of $G$ isomorphic to $F$ contains a vertex
$v$, then that subgraph may appear in one of possibly many {\em configurations},
depending on which vertex of $F$ coincides with $v$. For example, if $F$ is a
path of length 2, say a path with edges $ab$ and $bc$, then a subgraph
isomorphic to $F$ that contains vertex $v$ of $G$ may have either vertex $b$
coinciding with $v$ or vertex $a$ (or equivalently, vertex $c$) coinciding with
$v$. Thus we may write $s(F,G^v) = s(F^a,G^v) + s(F^b,G^v)$, where $s(F^a,G^v)$
and $s(F^b,G^v)$ are the contributions of the two configurations in which $F$
occurs as a subgraph in $G$. Our goal is to separately enumerate these
contributions to $s(F,G^v)$ (or $i(F,G^v)$), with the hope that such an
enumeration would help in proving reconstruction results for classes of graphs
or simplify the proofs of known results.

We now make these notions precise. First we define isomorphism of rooted
graphs. Let $F_1^x$ and $F_2^y$ be graphs $F_1$ and $F_2$, considered to be
rooted at $x$ and $y$, respectively. We say that they are \emph{isomorphic} if
there is an isomorphism $f\colon V(F_1)\to V(F_2)$ such that $f(x) = y$. For a
graph $F$ rooted at $x$, we denote by $s(F^x, G^v)$ the number of rooted
subgraphs of $G^v$ that are isomorphic to $F^x$ such that the root $x$ of the
subgraph coincides with $v$, and by $i(F^x, G^v)$ the number of induced rooted
subgraphs of $G^v$ that are isomorphic to $F^x$ such that the root $x$ of the
subgraph coincides with $v$.

We say that two vertices $u$ and $v$ in a graph $F$ are \emph{similar}, written
as $u\approx v$, if there exists $g \in \Aut(F)$ such that $v = g(u)$. Note that
$\approx$ is an equivalence relation on $V(F)$, and the equivalence classes of
the relation are the orbits of the action of $\Aut(F)$ on $V(F)$.

\begin{lem}
  Let $V_i,\cdots,V_s$ be the orbits of the action of $\Aut(F)$ on $V(F)$. Let
  $U \subseteq V(F)$ consist of one representative vertex from each $V_i$.  Then
  we have
  \begin{equation}
    \label{eq:i}
    s(F,G^v)=\sum_{u \in U} s(F^u, G^v)
  \end{equation}
  and
  \begin{equation}
    i(F,G^v)=\sum_{u \in U} i(F^u, G^v).
 \end{equation}
\end{lem}

We show with examples that individual contributions $s(F^u,G^v)$ are not
reconstructible in general. Our examples make use of the notion of
pseudo-similarity.  Two vertices $u$ and $v$ in a graph $G$ are
\textit{pseudo-similar} if the vertex-deleted subgraphs $G-v$ and $G-u$ are
isomorphic but $u$ and $v$ are not similar.

\begin{ex} \label{ex:1} Figure~\ref{fig:kocay} shows a graph $G$ with
  pseudo-similar vertices $u$ and $v$; this example is from \cite{gk1982}. We
  consider a hypomorphism $f\colon V(G) \to V(G)$ such that $f(u) = v$ and
  $f(v) = u$.

  Now let $F$ be the graph $\grjj$. Observe that the vertex set of $F$
  partitions in to three orbits, namely, the vertex $x$ of degree 1, the two
  vertices of degree 2 (which are similar, so let one of them be $y$), and the
  vertex $z$ of degree 3. We have
  \begin{align*}
    s(F^{x}, G^v)=0 & \text{ and } s(F^{x}, G^u)=1 \\
    s(F^{y}, G^v)=1 & \text{ and } s(F^{y}, G^u)=0 \\
    s(F^{z}, G^v)=1 & \text{ and } s(F^{z}, G^u)=1
  \end{align*}

  These observations are consistent with our calculations in Section
  \ref{sec:calculations}, where we show that $s(F^{x}, G^w) + s(F^{y}, G^w)$ is
  reconstructible for each $w \in V(G)$.  Similarly, if $F = P_4$ (the path on 4
  vertices) and $x$ is a vertex of degree 1 in $F$, then $s(F^x,G^v)=3$ and
  $s(F^x,G^u) = 3$, but $i(F^x,G^v)=2$ and $i(F^x,G^u)=1$.

\end{ex}

\begin{figure}[h]
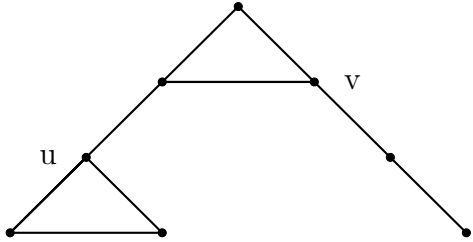
 \centering \griin
  \caption{An example from \cite{gk1982} of a graph $G$ with pseudo-similar
    vertices $u$ and $v$.}
    \label{fig:kocay}
\end{figure}

The limitation described here is similar to the degree-pair and degree-pair
sequence reconstruction for digraphs. We know that in general the degree-pair of
a deleted vertex cannot be reconstructed, but the degree-pair sequence of the
graph can be reconstructed. See Manvel \cite{manvel.1973} and Stockmeyer
\cite{stockmeyer.1977}. So we ask if the multiset $\{G_k^v, v\in V(G)\}$ could
be reconstructed, where $G_k^v$ is the subgraph of $G$ rooted at $v$ induced by
the vertices at distance at most $k$ from $v$. We partially answer this question
in Section~\ref{sec-vrc}, and we prove a similar result for edge reconstruction
in Section~\ref{sec-erc}.

\section{Refining Kelly's lemma for vertex reconstruction} \label{sec-vrc}

\begin{lem}
  Let $x \in V(G)$. Denote by $\Orbit_G(x)$ the set of vertices of $G$ that are
  similar to $x$. We denote by $s(F^x, G)$ the number of rooted subgraphs of $G$
  that are isomorphic to $F^x$. We have
  \[
    s(F^x,G) = \card{\Orbit_F(x)}s(F,G).
  \]
\end{lem}

The \emph{distance} between two vertices in a graph is the number of edges in a
shortest path between the two vertices.  The minimum among all the maximum
distances between a vertex and all other vertices is called the \emph{radius} of
the graph $G$, which we denote by $r(G)$.

Let $G$ be a connected graph. Let $v\in V(G)$. Let $G_k^v$ denote the subgraph
of $G$ rooted at $v$, induced by vertices at distance at most $k$ from $v$. Let
$S_k(G)$ denote the multiset $\{G_k^v, v\in V(G)\}$. It is known that if
$v(G) > 2$, then the degree sequence of $G$ is reconstructible; also for each
card $G-v$ in the deck, the degree of $v$ and the neighbourhood degree sequence
of $v$ are reconstructible. But we do not know if $S_1(G)$ is reconstructible,
and for a given card $G-v$, we cannot in general reconstruct $G^{v}_1, G^{v}_2$
as shown by the examples of graphs containing pseudo-similar vertices. The
following proposition partially answers the question of constructing $S_k(G)$.

\begin{prop} \label{pro:1} If $G$ is a connected graph with radius more that
  $k$, then $S_k(G)$ is reconstructible.
\end{prop}

\begin{proof} If the radius $r(G)$ of $G$ is more than $k$ then for all
  $v\in V(G)$, the graph $G^v_k$ has fewer than $v(G)$ vertices. Hence we claim
  that $S_k(G)$ is a subset of $\bigcup_{v\in V(G)} S_k(G-v)$ (taken as a
  multiset union), and the latter is reconstructible. Let $S$ be the set of
  distinct rooted graphs in $\bigcup_{v\in V(G)} S_k(G-v)$. In the following, we
  determine which members of $S$ are in $S_k(G)$ along with their
  multiplicities.

  Let $A^u \in S$. Let $n(A^u)$ be the number of vertices $v\in V(G)$ such that
  $G_k^v \cong A^u$. We want to prove that $n(A^u)$ is reconstructible. We have
  \begin{equation}\label{eq:a}
    s(A^u,G)=\displaystyle\sum_{B^w \in S}s(A^u,B^w)n(B^w).
  \end{equation}

  Since $s(A^u, G)=|\Orbit_A(u)|s(A,G)$ and $v(A) < v(G)$, we can reconstruct
  $s(A^u, G)$.  If $A^u$ is a maximal element in $S$ (i.e., $s(A^u,B^w)= 0$ for
  all $B^w \not\cong A^u$), then $s(A^u, G)=n(A^u)$. Now we order graphs in $S$
  as $A_1^{u_1}, A_2^{u_2}, \ldots$ so that $|E(A_i)|\geq |E(A_j)|$ for
  $i<j$. We can then solve Equation~\eqref{eq:a} recursively for each
  $A_i^{u_i}$ in the order $i=1, 2, \ldots$. Thus, we can reconstruct $S_k(G)$.
\end{proof}

We do not know if radius or diameter are in general reconstructible
parameters. But the result can be applied to bounded degree graphs, graphs
containing a vertex of degree 1, and possibly some other classes of graphs for
which estimates for the radius can be made from the deck.

\begin{ex}\label{ex:2} Let $G$ be the graph in
  Figure~\ref{fig:kocay}.  The rooted graphs $G_2^u$ and $G_2^v$ are not
  isomorphic. We reconstruct the multiset $S_2(G)$ by the Equation~\eqref{eq:a}
  counting $n(G_2^w)$ where $w\in V(G)$ (see Section \ref{sec:calculations}).

  Now, let $F$ be the graph \grjj\hspace{0.05cm} with a root vertex $y$ of
  degree 2. We have seen that the parameter $s(F^y,G^v)$ is not
  reconstructible. But, since $S_2(G)$ is reconstructible, the multiset
  $\{s(F^y,G^v)\mid v\in V(G)\}$ is reconstructible.
\end{ex}

\section{Refining Kelly's lemma for edge reconstruction} \label{sec-erc}
\begin{lem} \label{lem-kelly-edge } (Kelly's Lemma - edge version) If a graph
  $F$ is such that $e(F)< e(G)$, then $s(F,G)$ is 
  edge-reconstructible.
\end{lem}

Let $G$ be a graph. Let $e_1$ and $e_2$ be two edges in $G$. We define the
\textit{distance} between $e_1$ and $e_2$ to be the number of edges on a minimal
path that contains $e_1$ and $e_2$. If $e_1$ and $e_2$ are in different
components, then we define the distance between them to be infinity.

Let $e\in E(G)$. Let $G_{k}^e$ denote the subgraph of $G$ rooted at $e$ (i.e.,
with a distinguished edge $e$) induced by edges at distance at most $k$ from
$e$. Let $T_k(G)$ denote the multiset $\{G_{k}^e, e\in E(G)\}$.

Edges $a$ and $b$ of a graph $G$ are {\em similar} if there exists an
automorphism of $G$ that maps the ends vertices of $a$ to the ends vertices of
$b$, and are {\em pseudo-similar} if $G-a\cong G-b$, but $a$ is not similar to
$b$ in $G$.
\begin{ex} Let $G$ be the graph in Figure~\ref{fig:poirier}.  The edges $a$ and
  $b$ are pseudo-similar. Let $G^{a}_4$ and $G^{b}_4$ be two elements in
  $T_4(G)$, we have $G_4^a \not \cong G_4^b$.
\end{ex}

\begin{figure}[h]
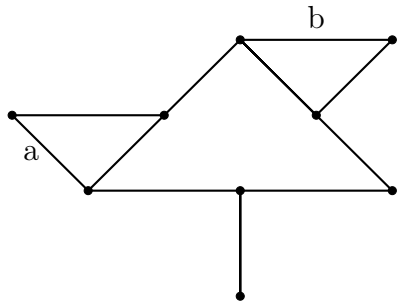

  \centering
  \griio
  \caption{An example of a pair of pseudo-similar edges $a$ and $b$ in a graph
    from Poirier \cite{poirier}.}
  \label{fig:poirier}
\end{figure}

\begin{prop}\label{prop:edge}
  If $G$ is a connected graph with radius more than $k > 1$, then $T_k(G)$ is
  edge-reconstructible.
\end{prop}

\begin{proof} Let $e \in E(G)$, and let $v$ be an end vertex of $e$. The
  distance between $v$ and any other vertex in $G^e_k$ is at most $k$. Thus, if
  the radius of $G$ is more than $k$ then every element in $T_k(G)$ has fewer
  than $v(G)$ vertices. Since $G$ is connected, we have $e(G^e_k) <
  e(G)$. Hence, $T_k(G)$ is a subset of $\bigcup_{e\in E(G)}T_k(G-e)$. Let $T$
  be the set of distinct elements in $\bigcup_{e\in E(G)}T_k(G-e)$. Now, we will
  determine which members of $T$ are in $T_k(G)$ along with their
  multiplicities.

  Let $A^e\in T$ . Let $s(A^e,G)$ denote the number of edge-rooted subgraphs of
  $G$ that are isomorphic to $A^e$. We have $s(A^e,G)=|\Orbit_A(e)|s(A,G)$,
  where $\Orbit_A(e)$ is the set of edges of $A$ that are similar to $e$. Let
  $m(A^e)$ be the number of edges $f\in E(G)$ such that $G_{k}^f\cong A^e$. We
  have
\begin{equation}\label{eq:b}
    s(A^e,G)=\displaystyle\sum_{B^f\in T}s(A^e,B^f)m(B^f).
\end{equation}

The term on the left-hand side of Equation~\eqref{eq:b} is edge
reconstructible. If $A^e$ is a maximal element in $T$, then $s(A^e,
G)=m(A^e)$. We order graphs in $T$ as $A_1^{e_1}, A_2^{e_2},\ldots$ so that
$e(A_i)\geq e(A_j)$ for $i<j$. We can then solve Equation~\eqref{eq:b}
recursively for each $A_i^{e_i}$ in the order $i=1,2,\ldots$.
\end{proof}

\begin{rem}
  Given $e\in E(G)$, we have $r(G-e)\geq r(G)$. We say that $G$ is
  \textit{radius-minimal} if for every edge $e$ of $G$ we have $r(G-e)>r(G)$. A
  connected graph $G$ is radius-minimal if and only if $G$ is a tree (see
  Walikar \cite{walikar}). Thus, $r(G)=\displaystyle\min_{e\in E(G)}r(G-e)$ when
  $G$ is not a tree. Hence, $r(G)$ is edge reconstructible.
\end{rem}

\section{Calculations for rooted graphs with a small number of
  vertices}\label{sec:calculations}

The following calculations for small graphs illustrate the difficulties in
counting the number of (induced) rooted subgraphs.

Let $F^x$ and $G^v$ be rooted graphs.  If $F^x\in\{\graa,\gree,\grff,\grqq\}$,
where the root vertex $x$ is marked by a bold dot, then the parameters
$i(F^x,G^v)$ and $s(F^x,G^v)$ are reconstructible (since the underlying unrooted
graph $F$ is vertex transitive). We obtain the following equations for rooted
graphs with a small number of vertices. Let $d_v(G)$ denote the degree of $v$ in
$G$.

\begin{align}
  i(\graa,G^v)&=i(\graab,G^v)=d_v(G)\label{1} \\
  i(\grcc, G^v)+i(\gree,G^v)&= s(\grcc,G^v)=\binom{d_v(G)}{2} \label{2} \\
  s(\grtt,G^v)&=\binom{d_v(G)}{3} = i(\grtt, G^v) +i(\grll,G^v) + i(\grhh, G^v)+i(\grqq, G^v) \label{3} \\ 
  s(\grll,G^v)&=i(\gree,G^v)(d_v(G)-2)= i(\grll,G^v)+2 i(\grhh, G^v)+ 3i(\grqq, G^v) \label{4} \\
  s(\grtt,G^v) &= i(\grtt,G^v)+i(\grll,G^v) + i(\grhh,G^v)+i(\grqq,G^v) \label{5a} \\
  s(\gruu,G^v) & = i(\gruu, G^v)+i(\grkk, G^v)+i(\grmm, G^v)+i(\grhh,G^v)+2i(\grgg, G^v) \nonumber\\ 
              & \quad + 3i(\grqq,G^v)\label{5b} \\
  s(\grpp,G^v)&=s(\grbb,G^v)(d_v(G)-1) - 2i(\gree,G^v)\nonumber \\ 
              &=i(\grpp,G^v)+2i(\grff,G^v)+i(\grkk,G^v)+2i(\grgg,G^v)  \nonumber\\
              &\quad + 4i(\grhh,G^v)+ 6i(\grqq,G^v)+2i(\grll,G^v)\label{6} \\ 
  s(\groo,G^v)& = i(\groo,G^v)+2i(\grff,G^v)+2i(\grmm,G^v) \nonumber \\
              & \quad + i(\grkk,G^v)+2i(\grhh,G^v) +4i(\grgg,G^v)+6i(\grqq,G^v)  \label{7}
\end{align}

By Equation~\eqref{2}, $i(\grcc,G^v)$ and $s(\grcc,G^v)$ are reconstructible,
hence by Equation~\eqref{eq:i}, $i(\grbb,G^v)$ and $s(\grbb,G^v)$ are
reconstructible. Thus, for all $F^x$ with $v(F)\leq 3$ and $G^v$ with
$v(G)\geq4$, we can calculate $i(F^x,G^v)$ and $s(F^x,G^v)$ from the deck of
$G$.

For rooted graphs with four vertices, we have already given an example where the
parameters $i(\grmm,G^v)$ and $i(\groo,G^v)$ are not reconstructible. By
Equation~\eqref{4}, we can calculate $s(\grll,G^v)$, thus
$s(\grkk, G^v)+s(\grmm, G^v)$ is reconstructible.  But we do not know how to
calculate $i(\grll,G^v)$. Also, by Equations~\eqref{6} and~\eqref{eq:i},
$s(\grpp,G^v)$ and $s(\groo,G^v)$ are reconstructible. But, we have already
shown that $i(\groo,G^v)$ is not reconstructible.

The next example illustrates Proposition \ref{pro:1}.

\begin{ex}
  Consider the graph $G$ in Figure~\ref{fig:kocay} again. We calculate $S_2(G)$
  using Proposition \ref{pro:1}.  We have
  $S=\{A_1^{u_1}, A_2^{u_2}, \ldots, A_{14}^{u_{14}}\}$, where the elements of
  $S$ are the following rooted graphs, respectively,
  \[
    \griinb,\griinc,\griind,\griine,
    \griinf,\griing,\griinh,\griini,\griinj,\griink,\griinl,\griino,\griinm,\griinn.
  \]
  Now we use Equation~\eqref{eq:a} recursively:
  \[
    n(A_1^{u_1})=s(A_1^{u_1},G)=1
  \]
  \[
    n(A_2^{u_2})=s(A_2^{u_2},G)-s(A_2^{u_2},A_1^{u_1})n(A_1^{u_1})=2-1\cdot 1=1
  \]
  \[
    n(A_3^{u_3})=s(A_3^{u_3},G)-s(A_3^{u_3},A_1^{u_1})n(A_1^{u_1})-s(A_3^{u_3},A_2^{u_2})n(A_2^{u_2})=5-3\cdot1-2\cdot1=0.
\]

\[
  \vdots
\]

\[
  n(A_{14}^{u_{14}})=s(A_{14}^{u_{14}},G)-\displaystyle\sum_{i=1}^{13}
  s(A_{14}^{u_{14}},A^{u_i}_i)n(A^{u_i}_i)=0.
\]
From the above equations we obtain
\[
  S_2(G)=\left\{\griinb,\griinc,\griine,\griing,\griinh,\griini,\griini,\griinm\right\}.
\]
\end{ex}

\section{Acknowledgements} 
\label{sec:acknowledgements}
The first author was supported by a doctoral scholarship from CAPES, Brasil. The
second author is supported by FAPEMIG (Fundação de Amparo à Pesquisa do Estado
de Minas Gerais, Brasil) (2023-2025), Process no.  APQ-02018-22. We would like
to thank our respective funding agencies.

\end{document}